\documentclass[12pt, reqno]{amsart}
\usepackage{epstopdf}
\usepackage{amssymb}
\usepackage[mathcal]{euscript}
\usepackage{upgreek}
\usepackage{marginnote}

\usepackage{amsmath, amsthm, amscd, amsfonts}
\usepackage{mathrsfs}
\usepackage{bbm, bm}
\usepackage[left= 3.4 cm, right= 3.4 cm, top= 3 cm, bottom=2.6 cm, foot= 1.1 cm, marginparwidth=2.7 cm, marginparsep=0.5 cm]{geometry}
\usepackage{abstract}
\usepackage{color}
\DeclareFontFamily{OT1}{pzc}{}
\DeclareFontShape{OT1}{pzc}{m}{it}{<-> s * [1.200] pzcmi7t}{}
\DeclareMathAlphabet{\mathpzc}{OT1}{pzc}{m}{it}
	
	\usepackage{hyperref}
	
	\newcommand{\mb}{\mathbb}
	
	\newcommand{\mf}{\mathfrak}
	\newcommand{\mc}{\mathcal}
	\newcommand{\ms}{\mathscr}
	\newcommand{\ov}{\overline}
	
	\newcommand{\wt}{\widetilde}
	\newcommand{\mz}{\mathpzc}
	\makeatletter 
	\@addtoreset{equation}{section}
	\makeatother  

	\numberwithin{equation}{section}

\begin{document}

\title{Correlation functions of gauged linear $\sigma$-model}

\author[Tian]{Gang Tian}
\address{
Department of Mathematics \\
Princeton University \\
Fine Hall, Washington Road \\
Princeton, NJ 08544 USA
}
\email{tian@math.princeton.edu}

\author[Xu]{Guangbo Xu}
\address{
Department of Mathematics \\
University of California, Irvine\\
Irvine, CA 92697 USA 
}
\email{guangbox@math.uci.edu}

\date{\today}
	
	\newtheorem{thm}{Theorem}[section]
	\newtheorem{lemma}[thm]{Lemma}
	\newtheorem{cor}[thm]{Corollary}
	\newtheorem{prop}[thm]{Proposition}
	\newtheorem{conj}[thm]{Conjecture}
	
	\theoremstyle{definition}
	\newtheorem{defn}[thm]{Definition}
	
	\theoremstyle{remark}
	\newtheorem{rem}[thm]{Remark}
	\newtheorem{hyp}[thm]{Hypothesis}
	\newtheorem{example}[thm]{Example}
	
\setcounter{tocdepth}{1}

\maketitle

\begin{abstract}
This is the second paper in a series following \cite{Tian_Xu}, on the construction of a mathematical theory of the gauged linear $\sigma$-model (GLSM). In this paper, assuming the existence of virtual moduli cycles and their certain properties, we define the correlation function of GLSM for a fixed smooth rigidified $r$-spin curve. 
\end{abstract}

\tableofcontents

\section{Introduction}

The gauged linear $\sigma$-model (GLSM) was introduced by Witten in \cite{Witten_LGCY} in physics background towards the understanding of the Landau-Ginzburg/Calabi-Yau correspondence. In the A-model, the close-string Calabi-Yau theory is understood as counting holomorphic curves (Gromov-Witten theory). Motivated from Gromov's pioneering work \cite{Gromov_1985} and Witten's interpretation \cite{Witten_sigma_model}, the foundation of Gromov-Witten theory were built up by \cite{Ruan_96}, \cite{Ruan_Tian}, \cite{Li_Tian}, \cite{Fukaya_Ono} in the setting of symplectic geometry. Numerous work has appeared and it has become a fundamental tool in symplectic geometry as well as algebraic geometry. On the other hand, the close-string Landau-Ginzburg theory has been constructed just recently, by Fan-Jarvis-Ruan (\cite{FJR1}, \cite{FJR3}, \cite{FJR2}) following Witten's idea (see \cite{Witten_spin}). 

The GLSM unifies the two theories under one framework. It has been very influential in physics but is yet to be constructed rigorously in mathematics. This paper is the second input in a series in which we are trying to build a mathematical theory of GLSM following Witten's proposal in \cite{Witten_LGCY}.

The core in our construction is the analysis of the moduli spaces of the classical equation of motion, which we called the gauged Witten equation. In our first paper \cite{Tian_Xu}, we set the gauged Witten equation under an appropriate framework. Suppose $X_0$ is a noncompact K\"ahler manifold admitting a holomorphic ${\mb C}^*$-action, and $Q: X_0 \to {\mb C}$ is a nondegenerate homogeneous holomorphic function. A typical example of $Q$ is a nondegenerate quintic polynomial on ${\mb C}^5$. Then the superpotential of the GLSM is $W = pQ : X_0 \times {\mb C} \to {\mb C}$, which we call a superpotential of Lagrange multiplier type. Then $W$ is invariant under a ${\mb C}^*$-action on $X = X_0 \times {\mb C}$. The triple $(X, W, {\mb C}^*)$ is the ``target space'' of GLSM. On the other hand, the domain of the GLSM is a rigidified $r$-spin curve, denoted by $\vec{\mc C}$, which is a punctured Riemann surface with some additional structures. Then the gauged Witten equation is an elliptic system about a gauge field $A$ and a matter field $u$ over $\vec{\mc C}$. The details are recalled in Section \ref{section3}.

In \cite{Tian_Xu}, we also studied several crucial analytical properties of gauged Witten equation and its moduli spaces. Among them, the most crucial one is the compactness of the moduli space of solutions to the perturbed gauged Witten equation over any fixed smooth rigidified $r$-spin curve. The next crucial ingredient is the transversality of the moduli space, which can be stated as 
\begin{thm}
For any strongly regular perturbation $\vec{P}$, and any asymptotic data $\vec\upkappa$, for any homology type $B$ of solutions (see Section \ref{section3} for precise meanings), the moduli space ${\mc M}\left( \vec{{\mc C}}; B, \vec\upkappa \right)$ of gauge equivalence classes of solutions to the $\vec{P}$-perturbed gauged Witten equation over $\vec{C}$, whose asymptotics are described by $\vec\upkappa$ and whose homology classes are prescribed by $B$, is compact and admits a virtual fundamental class
\begin{align*}
\left[ {\mc M}\left( \vec{{\mc C}}; B, \vec\upkappa \right) \right]^{vir} \in H_* \left( {\mc M}\left( \vec{\mc C}; B, \vec\upkappa \right) ; {\mb Q} \right).
\end{align*}
\end{thm}
The virtual fundamental class can be constructed by known techniques, such as the techniques developed in \cite{Fukaya_Ono} and \cite{Li_Tian}. It will be done in the incoming paper \cite{Tian_Xu_3}. We also remark that it is possible to use concrete perturbations of the gauged Witten equation to achieve transversality. 

In the scope of the current series, we only consider the moduli space for a fixed smooth $r$-spin curve with a rigidification. We have the associated virtual count
\begin{align}\label{equation11}
\# {\mc M}\left( \vec{\mc C}; B, \vec\upkappa \right) \in {\mb Q}
\end{align}
which is defined to be zero if the degree of the virtual fundamental cycle is nonzero.

The correlator is defined as a family of multi-linear maps on certain state space ${\ms H}_Q$ (see Section \ref{section2}). ${\ms H}_Q$ can be viewed as a generalization of both the state space in Landau-Ginzburg A-model and the state space in gauged Gromov-Witten theory. If $\vec{C}$ has $m$ marked points, then certain linear combinations of virtual counts (\ref{equation11}) give the correlation function (see Section \ref{section3}). 
\begin{align*}
\left\langle \ \cdot, \cdots, \cdot \ \right\rangle_{\vec{C}}^B: {\ms H}^{\otimes m}_Q \to {\mb Q}.
\end{align*}
The coefficients of the linear combinations as well as the virtual counts (\ref{equation11}) depend on the choice of a strongly regular perturbation. However, we have
\begin{thm}
The correlation function is independent of the choice of ``strongly regular '' perturbations and various other choices.
\end{thm}
The proof is basically a cobordism argument. In many similar situation, such as Donaldson theory and Gromov-Witten theory, one can actually prove a direct cobordism of moduli spaces obtained by choosing different auxiliary data, such as the metric or the almost complex structure (cf. \cite{Donaldson_Kronheimer}, \cite{McDuff_Salamon_2004}). In cases where virtual techniques are used, one can prove a cobordism in the virtual sense (cf. \cite{Fukaya_Ono}, \cite{Li_Tian}, \cite{Mundet_Tian_Draft}). In the current situation, the moduli spaces for different strongly regular perturbations may not be cobordant directly, but bifurcations (wall-crossings) happen in the interior of the cobordism. Such bifurcation analysis were carried out in many cases, such as \cite{Floer_intersection} and (closest to our situation) \cite{FJR3}. Then the difference between the virtual counts (\ref{equation11}) for two sets of perturbations is given by a wall-crossing formula (see Theorem \ref{thm45}). The coefficients of the linear combinations in defining the correlatino functions also differ by opposite wall-crossing terms, which exactly make the correlation function invariant. The detailed proof of the wall-crossing formula is given in \cite{Tian_Xu_3}.

\subsubsection*{Organization of the paper}

In Section \ref{section2} we define the state spaces in our formulation of GLSM. In Section \ref{section3} we define the correlation function, assuming the existence of the virtual cycles. In Section \ref{section4} we list the properties of the virtual cycles which are necessary to derive the well-definedness of the correlation functions. 

\subsubsection*{Acknowledgements}

We would like to thank Simons Center for Geometry and Physics for hospitality during our visit in summer 2013. We would like to thank Kentaro Hori, David Morrison, Edward Witten for useful discussions on GLSM. The second author would like to thank Chris Woodward for helpful discussions.

\section{The state space for GLSM}\label{section2}

The inputs of the correlation function we are going to define are classes (states) in certain cohomology groups (the state space). We have analogues in previously studied theories. In Gromov-Witten theory, the state spaces are the ordinary cohomology groups of symplectic manifolds; in gauged Gromov-Witten theory, they are equivariant cohomologies; in Fan-Jarvis-Ruan's Landau-Ginzburg A-model (see \cite{FJR2}), they are certain cohomology groups naturally associated with the singularity.  

\subsubsection*{Lagrange multipliers}

Let $(X_0, \omega, J)$ be a noncompact K\"ahler manifold. Assume that there is a holomorphic ${\mb C}^*$-action which restricts to a Hamiltonian $S^1$-action on $X_0$. For every $\upgamma \in S^1$, let $X_{0, \upgamma} \subset X_0$ be the fixed point set of $\upgamma$ and let $N_{0, \upgamma} \to X_{0, \upgamma}$ be the normal bundle. Suppose $Q: X_0 \to {\mb C}$ is a holomorphic function. We assume that $Q$ is homogeneous of degree $r$. This means that 
\begin{align*}
\forall \xi \in {\mb C}^*,\ x\in X_0,\ Q(\xi x) = \xi^r Q(x).
\end{align*}
For any $a \in {\mb C}$, let $Q^a:= Q^{-1}(a) \subset X_0$. Then ${\mb Z}_r$ acts on $Q^a$. We make the following assumptions on $Q$.
\begin{hyp}
\begin{enumerate}
\item[({\bf Q1})] $Q$ has a unique critical point $\bigstar \in X_0$ (the critical value must be zero).

\item[({\bf Q2})] There exist a constant $c_Q>1$ and $G$-invariant compact subset $K_0 \subset X_0$ such that
\begin{align*}
x\notin K_0\Longrightarrow {1\over c_Q} \left| \nabla^3 Q \right| \leq \left| \nabla^2 Q \right| \leq c_Q \left| \nabla Q \right|.
\end{align*}
Moreover, for every $\delta>0$, there exists $c_Q(\delta)>0$ such that
\begin{align*}
d(x, Q^0)\geq \delta,\ x\notin K_0\Longrightarrow |\nabla Q(x)| \leq c_Q(\delta) |Q(x)|.
\end{align*}

\item[({\bf Q3})] For every $\upgamma \in {\mb Z}_r$, it is easy to see that $dQ$ vanishes along the normal bundle $X_{0, \upgamma}$. We assume that $\nabla^2 Q$ vanishes along $N_{0, \upgamma}$.
\end{enumerate}
\end{hyp}

\begin{rem}
The above hypothesis was assumed in \cite{Tian_Xu}. ({\bf Q2}) is necessary to guarantee the compactness of the moduli space. ({\bf Q3}) is not essential and can be removed. These conditions are satisfied nondegenerate quasi-homogeneous polynomials on ${\mb C}^N$.
\end{rem}

\subsubsection*{GLSM State space}
	
The state space consists of narrow sectors and broad sectors. 
\begin{defn}
$\upgamma \in {\mb Z}_r$ is {\bf broad} (resp. {\bf narrow}) if the restriction $Q_\upgamma:= Q|_{X_{0, \upgamma}}$ doesn't (resp. does) vanish identically. 
\end{defn}

It is easy to prove that $\bigstar\in X_{0, \upgamma}$ for each $\upgamma$ and for broad $\upgamma$, still the unique critical point of $Q_\upgamma$.
	
From now on we will introduce many homology and cohomology groups. Whenever the coefficient ring is omitted, we mean integral homology or cohomology.
\begin{defn}
If $\upgamma$ is narrow, then the (reduced) $\upgamma$-sector of the GLSM state space is a 1-dimensional ${\mb Q}$-vector space, generated by an element $e_\upgamma$. If $\upgamma$ is broad, then the ({\it reduced}) $\upgamma$-sector of the GLSM state space is
\begin{align*}
{\ms H}_\upgamma:= H^{n_\upgamma-1} \left( Q_\upgamma^a; {\mb Q} \right)^{{\mb Z}_r}.
\end{align*}
Here $n_\upgamma = {\rm dim}_{{\mb C}} X_{0, \upgamma}$ and $a\in {\mb C}^*$ is an arbitrary regular value of $Q_\upgamma$. The total state space is 
\begin{align*}
{\ms H}_Q:= \bigoplus_{\upgamma \in {\mb Z}_r} {\ms H}_\upgamma.
\end{align*}
\end{defn}

Here, to see that the broad sectors are independent of the choice of $a$, we consider the {\it monodromy action}, which is a linear isomorphism 
\begin{align*}
{\mf m}: H^* \left( Q_\upgamma^a \right) \to H^* \left( Q_\upgamma^a \right).
\end{align*}
This map is defined as a straightforward extension of the monodromy action for isolated singularities. Namely, let $C$ be a simple closed curve in ${\mb C}$ which passes through $a$ and avoids the origin (the singular value of $Q_\upgamma$). Then $Q_\upgamma^{-1}(C) \to C$ is a locally trivial fibration. This gives the monodromy action ${\mf m}$ on the integral homology of $Q_\upgamma^a$, which is independent of the choice of such simple closed curves.

\begin{lemma}
${\mf m}$ is equal to the action by the generator of ${\mb Z}_r$ on $H^*\left( Q^a \right)$. 
\end{lemma}
\begin{proof}
Consider the path $e^{{\bm i} \theta}$ for $\theta \in [0, {2\pi \over r}]$. Let $a(\theta) = e^{{\bm i} r \theta} a$. Then $e^{{\bm i} \theta}$ induces an isomorphism
\begin{align*}
H^*\left( Q^{a(\theta)}; {\mb Z} \right) \to H^* \left( Q^a; {\mb Z}\right)
\end{align*}
which is equal to the parallel transport along the arc between $a$ and $a(\theta)$. Since $t({2\pi \over r}) = t$, we see that the action by the generator of ${\mb Z}_r$ is equal to the parallel transport along a loop, which is exactly the monodromy action.
\end{proof}
Therefore, we see that the ${\mb Z}_r$-invariant part of $H^* \left( Q_\upgamma^a \right)$ is independent of the choice of $a\in {\mb C}^*$. Therefore the state space is well-defined. 

\begin{rem}
There is an enrichment of the state space by including more states for the broad sectors. This enrichment, formally, will be closer to the state space in \cite{FJR2} and the state space used in gauged Gromov-Witten theory. It will be mentioned at the end of this section.
\end{rem}

\subsubsection*{Vanishing cycles and Lefschetz thimbles}

We need a geometric description of generators of the homology group dual to the state space, i.e., vanishing cycles and Lefschetz thimbles. For each $\upgamma$, we can identify a neighborhood of $\bigstar$ in $X_{0, \upgamma}$ as a neighborhood of $0$ in ${\mb C}^{n_\upgamma}$. Denote this neighborhood by $B_\epsilon^{n_\upgamma}$. Then choose $\delta<< \epsilon$ and denote $U_\delta \subset {\mb C}$ the $\delta$-neighborhood of the origin. Then for $t \in \partial U_\delta$, denote $V_t:= B_\epsilon^{n_\upgamma} \cap Q_\upgamma^t$ and denote $V_T:= Q_\upgamma^{-1}(U_\delta) \cap B_\epsilon^{n_\upgamma}$. Then the classical result of Brieskorn \cite{Brieskorn_1970} says that the relative homology
\begin{align*}
H_{n_\upgamma} \left( V_T, V_t\right)
\end{align*}
is generated by Lefschetz thimbles.

Moreover, we have
\begin{lemma}
The inclusion $(V_T, V_t) \to \left(X_{0, \upgamma} , Q_\upgamma^t \right)$ induces an isomorphism
\begin{align*}
H_{n_\upgamma} \left( V_T, V_t \right) \simeq H_{n_\upgamma} \left( X_{0, \upgamma}, Q_\upgamma^t\right).
\end{align*}
\end{lemma}

\begin{proof}
Since $\bigstar$ is the only critical point, we see that $Q_\upgamma^{-1}(U_\delta)$ is a deformation retract of $X_{0, \upgamma}$. Therefore it suffices to prove the isomorphism
\begin{align*}
H_{n_\upgamma} \left( V_T, V_t \right) \simeq H_{n_\upgamma} \left( Q_\upgamma^{-1}(U_\delta), Q_\upgamma^t \right).
\end{align*}
Let $\check{V}_T: Q_\upgamma^{-1}(U_\delta) \setminus V_T$, $\check{V}_t:= Q_\upgamma^t \setminus V_t$. Then because the restriction $Q_\upgamma:\check{V}_T \to U_\delta$ has no critical point, it is a trivial fibration. Therefore for $t \in B_\delta$, we have a trivialization 
\begin{align*}
\check{V}_T \simeq U_\delta \times \check{V}_t.
\end{align*}
Therefore, we have homotopy equivalence
\begin{align*}
\left( Q_\upgamma^{-1}(U_\delta), Q_\upgamma^t \right) \sim \left( V_T \cup \check{V}_t, Q^t_\upgamma \right).
\end{align*}
Then by excision, we have
\begin{align*}
H_{n_\upgamma} \left( Q_\upgamma^{-1}(U_\delta), Q_\upgamma^t \right) \simeq H_{n_\upgamma} \left( V_T \cup \check{V}_t; Q_\upgamma^t \right) \simeq H_{n_\upgamma} \left( V_T, V_t \right).
\end{align*}
\end{proof}

So we say that the relative homology $H_{n_\upgamma}\left( X_{0, \upgamma}, Q_\upgamma^t \right)$ is generated by Lefschetz thimbles. Indeed by the local triviality of the fibration $X_{0, \upgamma}\setminus \{\bigstar\} \to {\mb C}^*$, we see that this is true not just for $t$ close to $0$, but all nonzero $t$. 

Now for each $a \in {\mb C}^*$, we have the exact sequence  
\begin{align*}
H_{n_\upgamma} (X_{0, \upgamma}) \to H_{n_\upgamma} \left(X_{0, \upgamma}, Q_\upgamma^a \right) \to H_{n_\upgamma-1} \left( Q_\upgamma^a \right) \to H_{n_\upgamma-1} (X_{0, \upgamma}). 
\end{align*}
We make the following simplifying assumption
\begin{hyp}\label{hyp27}
For any broad $\upgamma \in {\mb Z}_r$, the map $H_{n_\upgamma}\left( X_{0, \upgamma}, Q_\upgamma^a \right) \to H_{n_\upgamma-1} \left( Q_\upgamma^a \right)$ is an isomorphism. In other words, the middle dimensional homology of $Q_\upgamma^a$ is generated by vanishing cycles.
\end{hyp}
This hypothesis is clearly satisfied by quasihomogeneous polynomials on ${\mb C}^N$. 

\subsubsection*{Intersection pairing}

We need a natural perfect pairing between broad states in order to define the correlation function. For $M>>0$ sufficient large, denote
\begin{align*}
Q_\upgamma^\infty:= \left( {\rm Re} Q_\upgamma \right)^{-1} \left( [M, +\infty ) \right),\ Q_\upgamma^{-\infty}:= \left( {\rm Re} Q_\upgamma \right)^{-1} \left( (-\infty, -M]\right).
\end{align*}
Then we have a perfect pairing
\begin{align}\label{equation21}
H_{n_\upgamma} \left( X_{0,\upgamma}, Q_\upgamma^{-\infty} \right) \otimes H_{n_\upgamma} \left( X_{0, \upgamma}, Q_\upgamma^{+\infty} \right) \to {\mb Z}.
\end{align}
This pairing is described in \cite[Page 36]{FJR2} in the case of quasihomogeneous polynomials on ${\mb C}^N$ but for exactly the same reason we have it for the more general case. 

On the other hand, via the parallel transport, we have canonical isomorphisms
\begin{align}\label{equation22}
H_{n_\upgamma}\left( X_{0, \upgamma}, Q_\upgamma^a \right)^{{\mb Z}_r} \simeq H_{n_\upgamma} \left( X_{0, \upgamma}, Q_\upgamma^{\pm\infty} \right)^{{\mb Z}_r}.
\end{align}
Moreover, choose $\xi \in S^1$ such that $\xi^r = -1$. $\xi$ implies an isomorphism
\begin{align}\label{equation23}
H_{n_\upgamma} \left( X_{0, \upgamma}, Q_\upgamma^\infty \right)^{{\mb Z}_r} \to H_{n_\upgamma} \left( X_{0, \upgamma}, Q_\upgamma^{-\infty} \right)^{{\mb Z}_r},
\end{align}
which is independence of the choice of $\xi$ because we have restricted to the monodromy invariant part. Then by Hypothesis \ref{hyp27} and (\ref{equation21})--(\ref{equation23}), we have a perfect pairing
\begin{align*}
H_{n_\upgamma-1} \left( Q_\upgamma^a \right)^{{\mb Z}_r} \otimes H_{n_\upgamma-1} \left( Q_\upgamma^a \right)^{{\mb Z}_r}  \to {\mb Z}.
\end{align*}
Therefore, by the duality between homology and cohomology we have a canonical identification
\begin{align}\label{equation24}
{\ms H}_\upgamma \simeq H_{n_\upgamma-1} \left( Q_\upgamma^a; {\mb Q} \right)^{{\mb Z}_r}.
\end{align}

\subsubsection*{$\infty$-relative cycles in $Q_\upgamma^a$}

The use of Lagrange multiplier requires us to consider the complex Morse theory of the hypersurfaces $Q_\upgamma^a$. If we have a holomorphic Morse function $F$ defined on $Q_\upgamma^a$, then critical points of $F$ (together with rays emitting from it) represent certain $\infty$-relative cycles. We will use the intersection between compact cycles and $\infty$-relative cycles, which is described as follows. For any compact subset $K \subset X$, we can consider the relative homology $H_* \left( Q_\upgamma^a, Q_\upgamma^a \setminus K\right)$. The inverse limit with respect to the direct system of compact subsets under inclusion is denoted by 
\begin{align*}
H_* \left( Q_\upgamma^a, \infty \right).
\end{align*}
This is the dual space of $H^*_c\left( Q_\upgamma^a \right)$. Then we have the intersection pairing
\begin{align}\label{equation25}
\cap: H_* \left( Q_\upgamma^a \right) \otimes H_* \left( Q_\upgamma^a, \infty \right) \to {\mb Z}.
\end{align}


\begin{rem}
There is certain enrichment of the GLSM state space. Take $X = X_0 \times {\mb C}$ with an additional $K:= S^1$-action by
\begin{align*}
e^{{\bm i} \theta} (x, p) = ( e^{{\bm i} \theta}x, e^{- {\bm i} r \theta} p ).
\end{align*}
We take $W: X\to {\mb C}$ to be $W(x, p) = p Q(x)$, which is invariant under the $K$-action. Then for $\upgamma \in {\mb Z}_r$ and any $a \in {\mb C}^*$, denote $W_\upgamma^a = W_\upgamma^{-1}(a)$. The {\it enriched} $\upgamma$-sector of GLSM state space is defined by 
\begin{align*}
\wt{\ms H}_\upgamma:= H^*_K \left( X_\upgamma, W_\upgamma^a; {\mb Q} \right)^{{\mb Z}_r}.
\end{align*}
Here $(X_\upgamma, W_\upgamma^a)$ has the $K$-action and ${\mb Z}_r$-action commuting with each other. The enriched GLSM state space is defined as
\begin{align*}
\wt{\ms H}_Q:= \bigoplus_{\upgamma \in {\mb Z}_r} \wt{\ms H}_\upgamma.
\end{align*}
We see that formally the enriched state space generalizes the state space used in \cite{FJR2} (when $K$ is trivial), and the equivariant cohomology used in gauged Gromov-Witten theory (when $W \equiv 0$ and the ${\mb Z}_r$-action can be ignored) (see \cite{Cieliebak_Gaio_Mundet_Salamon_2002}). A more comprehensive correlation function can be defined over $\wt{\ms H}_Q$. 
\end{rem}

\section{Definition of the correlation function}\label{section3}

In this section we define the correlation function, as a collection of ${\mb Q}$-valued multi-linear function on ${\ms H}_Q$. The definition depends on the construction of the virtual fundamental class of the moduli space of solutions to the perturbed gauged Witten equation. The construction will be provided in a separate paper.

\subsection{Perturbed gauged Witten equation}

We recall the set-up of perturbed Witten equation given in \cite{Tian_Xu}. 

Let $({\mc C}, {\mc L}, \upvarphi)$ be a smooth $r$-spin curve, with orbifold markings ${\bm z} = (z_1, \ldots, z_m)$. Here ${\mc C}$ is a smooth orbifold Riemann surface, with possible nontrivial orbifold structures only at $z_1, \ldots, z_m$; ${\mc L} \to {\mc C}$ is a holomorphic orbifold line bundle; $\upvarphi$ is an isomorphism of orbifold line bundles
\begin{align*}
\upvarphi: {\mc L}^{\otimes r} \to {\mc K}_{\log}:= {\mc K}_{\mc C} \otimes {\mc O}(z_1) \otimes \cdots \otimes {\mc O}(z_m).
\end{align*}
$r$-spin structures are labelled by 
\begin{align*}
\vec{\upgamma}:= \left( \upgamma_1, \ldots, \upgamma_m \right) \in ({\mb Z}_r)^m.
\end{align*}
$\upgamma_i$ is called the monodromy of the $r$-spin structure at the marking $z_i$. The notion of $\upgamma$ being narrow or broad has been defined in Section \ref{section2}. A marking $z_i$ is called broad or narrow if its monodromy $\upgamma_i$ is broad or narrow respectively. In the current situation, we only consider a fixed $r$-spin curve, and we assume that the first $b$ markings are broad and the last $n = m- b$ markings are narrow.

Let $\Sigma$ be the smooth Riemann surface underlying ${\mc C}$. ${\bm z}$ is regarded as punctures on $\Sigma$ and $\Sigma^*:= \Sigma \setminus {\bm z}$. For each marking $z_i$, we fix a local holomorphic coordinate $w$ on $\Sigma$ centered at $z_i$. We assume these coordinate patches are disjoint from each other. A rigidification of the $r$-spin structure at $z_i$ is a choice of an element $e_j \in {\mc L}|_{z_i}$ such that 
\begin{align*}
\upvarphi( e_j^{\otimes r}) = {dw \over w} \in {\mc K}_{\log}|_{z_i}.
\end{align*}
We fix rigidifications $\vec{\upphi}= ( \upphi_1, \ldots, \upphi_m )$ at all punctures. Now we fix
\begin{align*}
\vec{\mc C}:= \left( {\mc C}, {\mc L}, \upvarphi; \vec{\upphi} \right)
\end{align*}
as a rigidified $r$-spin curve, which is the domain of the gauged Witten equation.

The line bundle ${\mc L}$ descends to an ordinary line bundle $L \to \Sigma^*$. In \cite{Tian_Xu} we define the notion of adapted Hermitian metrics, which is a class of $W_{loc}^{2, p}$-Hermitian metrics on $L$ compatible with the $r$-spin structure. We choose a smooth adapted metric $H_0$ on $L$, and denote by $Q_0\to \Sigma^*$ the unit circle bundle. 

On the other hand, choose another $K= S^1$-principal bundle $Q_1 \to \Sigma$, whose restriction to $\Sigma^*$ is still denoted by $Q_1$. Then in \cite{Tian_Xu}, we considered a space ${\mz A}$ of $G= S^1 \times S^1$-connections on $Q:= Q_0 \times_{\Sigma^*} Q_1$. Moreover, ${\mz G}$ is the space of gauge transformations $g: \Sigma^* \to G$ of class $W_{loc}^{2, p}$ such that $g$ is asymptotic to the identity in a $W_\delta^{2, p}$-manner for some $\delta>0$ ($\delta$ is allowed to change). 

$G= S^1 \times K$ acts on $X$. Denote $Y:= Q \times_G X \to \Sigma^*$ the fibre bundle. The $r$-spin structure induces a family of lifting
\begin{align*}
{\mc W}_A \in \Gamma \left( Y, \pi^* K_{\log} \right),\ A\in {\mz A}.
\end{align*}
The vertical tangent bundle $T^\bot Y \to Y$ admits a natural Hermitian metric. With respect to this metric, we have the vertical gradient
\begin{align*}
\nabla {\mc W}_A \in \Gamma \left( Y, \pi^* \Omega_\Sigma^{0,1} \otimes T^\bot Y \right).
\end{align*}
Choose a biinvariant metric on the Lie algebra ${\mf g}$ and an area form $\Omega$ on $\Sigma$. The gauged Witten equation reads 
\begin{align}\label{equation31}
\left\{ \begin{array}{ccc}
\ov\partial_A u + \nabla {\mc W}_A (u) & = & 0,\\
* F_A + \mu(u) & = & 0.
\end{array} \right.
\end{align}

\subsubsection*{Perturbations}

Whenever there is broad punctures, the linearization of (\ref{equation31}) (modulo gauge transformation) is not a Fredholm operator in a natural way. We have to perturb the equation near broad punctures. A perturbation is described as
\begin{align*}
\vec{P} = \left( \vec{a}, \vec{F} \right) = \left( a_i, F_i \right)_{i=1}^b.
\end{align*}
Here for each $i$, $a_i \in {\mb C}^*$ and $F_i: X_0 \to {\mb C}$ is a holomorphic function. Denote
\begin{align*}
W_i:= W - a_i p + F_i.
\end{align*}

\begin{defn}
$F_i$ is {\bf $\upgamma_i$-admissible} if the following are satisfied.
\begin{enumerate}
\item[({\bf P1})] There exist $r_l \leq {1\over 2} r$ ($l = 2, \ldots, s$) such that 
\begin{align*}
F_i = \sum_{l=2}^s F_i^{(l)}
\end{align*}
and $F_i^{(l)}: X_0 \to {\mb C}$ is a holomorphic function of degree $r_l$.

\item[({\bf P2})] There exists $c_i >0$ such that for $l = 2, \ldots, s$
\begin{align*}
\left| F_i^{(l)} (x) \right| \leq c_i  \left( 1 + |\mu_+(x)| \right)^{1\over 2},\ \left| d F_i^{(l)} (x) \right| \leq c_i.
\end{align*}
\end{enumerate}

$P_i:= (a_i, F_i)$ is called {\bf $\upgamma_i$-regular} if $F_i$ is $\upgamma_i$-admissible and
\begin{enumerate}
\item[({\bf P3})] for every $\epsilon \in (0, 1]$, the restriction of $W_{i, \epsilon} = pQ - \epsilon^r a_i p + \epsilon^r F_i^\epsilon$ to $X_{\upgamma_i}$ is a holomorphic Morse function. Here $F_i^\epsilon (x) = F_i(\epsilon^{-1} x)$.

\item[({\bf P4})] The perturbed functions $W_{i, \epsilon}$ has no critical points at infinity in the following sense: for every $T>1$, there is a $G$-invariant compact subset $K_T\subset X$ and $\epsilon_T>0$ such that 
\begin{align*}
\epsilon \in \left[ T^{-1}, 1\right],\ \left| \nabla W_{i, \epsilon} (x, p) \right| \leq \epsilon_T \Longrightarrow (x, p) \in K_T. 
\end{align*}
\end{enumerate}

$P_i= (a_i, F_i)$ is called {\bf $\upgamma_i$-strongly regular} if it is $\upgamma_i$-regular and all critical values of the restriction of $W_i$ to $X_{\upgamma_i}$  have distinct imaginary parts. 

$\vec{P}= \left( \vec{a}, \vec{F}\right)$ is called regular (resp. strongly regular ) if $(a_i, F_i)$ is $\upgamma_i$-regular (resp. $\upgamma_i$-strongly regular) for each broad puncture $z_i$ whose monodromy is $\upgamma_i$.
\end{defn}

\begin{lemma}\label{lemma32}
If $(p, x)$ is a critical point of $W_i|_{X_{\upgamma_i}}$, then for every $\epsilon>0$, $(p, \epsilon x)$ is a critical point of $W_{i, \epsilon}|_{X_{\upgamma_i}}$ and 
\begin{align*}
W_{i, \epsilon}(p, \epsilon x) = \epsilon^r W_i(p, x).
\end{align*}
\end{lemma}

\begin{proof}
If $(p, x) \in {\rm Crit} W_i|_{X_{\upgamma_i}}$, then
\begin{align*}
Q(x) = a_i,\ p dQ(x) + dF_i (x) = 0.
\end{align*}
Then $Q(\epsilon x) = \epsilon^r a_i$. Regard $\epsilon$ as the diffeomorphism of $X_0$ by multiplying $\epsilon$. Then
\begin{align*}
\left( \epsilon^* \left( pdQ + \epsilon^r dF_i^\epsilon \right) \right)(x) = \epsilon^r \left( pdQ(x) + dF_i(x) \right) = 0.
\end{align*}
Therefore $(p, \epsilon x) \in W_{i, \epsilon} |_{X_{0, \upgamma_i}}$. Moreover, the critical value is 
\begin{align*}
W_{i, \epsilon}(p, \epsilon x) = \epsilon^r F_i^\epsilon (\epsilon x) = \epsilon^r F_i(x) =  \epsilon^r W_i(p, x).
\end{align*}
\end{proof}

In applications, we may assume that certain class of $\upgamma_i$-admissible perturbations form a finite dimensional, nonzero complex vector space. This is the case when we consider the GLSM for a quasihomogeneous polynomial $Q: {\mb C}^N \to {\mb C}$, where the space of $F$'s is the space of all linear functions on ${\mb C}^N$. So we assume the following conditions.
\begin{hyp}\label{hyp33}
The space of all holomorphic functions $F: X_0 \to {\mb C}$ satisfying ({\bf P1}) and ({\bf P2}) is a finite dimensional nonzero complex vectors pace $V_F$. For fixed $a_i\in {\mb C}^*$, there is an analytic subset $V_F^{sing}(a_i)\subset V_F$ such that for every $F_i \in V_F \setminus V_F^{sing}(a_i)$, $(a_i, F_i)$ is $\upgamma_i$-regular. 
\end{hyp}

Suppose $P_i:= (a_i, F_i)$ is $\upgamma_i$-regular. Then there is a one-to-one correspondence between critical points of $F_i|_{Q_{\upgamma_i}^{a_i}}$ and critical points of $W_i|_{X_{\upgamma_i}}$. We use both of the two perspectives. A critical point is denoted by $\upkappa_i$. Moreover, critical points of $W_{i, \epsilon}|_{X_{\upgamma_i}}$ exist in smooth families parametrized by $\epsilon \in (0, 1]$. We also use $\upkappa_i$ to denote a family $\upkappa_i(\epsilon)$ of critical points of $W_{i, \epsilon}|_{X_{\upgamma_i}}$. We denote ${\rm Crit} W_{P_i}$ the set of all such families, which is a finite set. Denote
\begin{align*}
{\rm Crit} W_{\vec{P}} = \prod_{i=1}^b {\rm Crit} W_{P_i}.
\end{align*}
An element of it is denoted by $\vec{\upkappa}= (\upkappa_1, \ldots, \upkappa_b)$.

In this section, we fix a strongly regular perturbation $\vec{P} = (\vec{a}, \vec{F})$. As in \cite{Tian_Xu}, by choosing a cut-off function $\beta: \Sigma^* \to [0,1]$ supported near all broad punctures, we can lift the perturbation to $Y$. The lifting depends on the connection $A$, as well as choices of frames at broad punctures of $Q_1 \to \Sigma$. Denote by
\begin{align*}
\vec\uppsi:= \left( \uppsi_1, \ldots, \uppsi_b\right): K^b \to Q_1|_{\{z_1, \ldots, z_b\}}
\end{align*}
a choice of frames. The perturbed family of superpotentials is denoted by
\begin{align*}
\wt{\mc W}_A^{\vec\uppsi}\in \Gamma \left( Y, \pi^* K_{\log} \right).
\end{align*}

The perturbed gauged Witten equation is the following one on triples $(A, u, \vec{\uppsi})$
\begin{align}\label{equation32}
\left\{ \begin{array}{ccc}
\ov\partial_A u + \nabla \wt{\mc W}_A^{\vec{\uppsi}}(u) & = & 0;\\
 * F_A + \mu (u) & = & 0.
\end{array} \right.
\end{align}
This equation transform naturally under the action of the group of gauge transformations ${\mz G}$. 

The energy of pairs $(A,u)$ is defined as
\begin{align*}
E(A, u) = {1\over 2} \left( \left\| d_A u \right\|_{L^2}^2 + \left\| F_A \right\|_{L^2}^2 + \left\| \mu(u) \right\|_{L^2}^2 \right) + \left\| \nabla \wt{\mc W}_A^{\vec{\uppsi}}(u) \right\|_{L^2}^2
\end{align*}
where the norms are taken with respect to the metric on $\Sigma^*$ determined by $\Omega$ and the complex structure.

We summarize the main results of \cite{Tian_Xu} in the following theorem.
\begin{thm}
\begin{enumerate}
\item For any solution $(A, u)$ to (\ref{equation32}) with finite energy and $\left| \mu(u) \right|$ bounded on $\Sigma^*$ (such solutions are called bounded solutions), there exists $\upkappa_i \in X_{\upgamma_i}$ such that with respect to certain trivialization of $Y$ near $z_i$, 
\begin{align*}
\lim_{z \to z_i} u(z) = \upkappa_i.
\end{align*}
Moreover, if $\upgamma_i$ is broad, then $\upkappa_i\in {\rm Crit} \left( W_{i, \epsilon}|_{X_{\upgamma_i}} \right)$ for some $\epsilon \in (0, 1]$.

\item Any bounded solution defines a homology class $[A, u] \in \Gamma_X^G:= H_G^2 \left(X; {\mb Z}[r^{-1}] \right)$. There exists a function $E: \Gamma_X^G \to {\mb R}_+$ such that for every bounded solution $(A, u)$ to (\ref{equation32}) with $[A, u]= B \in \Gamma_X^G$, we have
\begin{align*}
E(A, u) \leq E(B).
\end{align*}

\item For every $E$, the moduli space of gauge equivalence classes of bounded solutions $(A, u)$ to (\ref{equation32}) satisfying $E(A, u) \leq E$ is compact up to degeneration of solitons at broad punctures. In particular, if the perturbation $\vec{P}$ is strongly regular, then the moduli space itself is compact.
\end{enumerate}
\end{thm}

Therefore, for any $B \in \Gamma_X^G$ and $\vec\upkappa  = (\upkappa_1, \ldots, \upkappa_b) \in {\rm Crit} W_{\vec{P}}$, denote by 
\begin{align}\label{equation33}
{\mc M}_{\vec{P}}\left( \vec{\mc C}, B, \vec\upkappa \right)
\end{align}
the moduli space of gauge equivalence classes of solutions to (\ref{equation32}) which represent the class $B \in \Gamma_X^G$ and such that for each broad puncture $z_i$, $i = 1, \ldots, b$, the limit of $u$ at $z_i$ belongs to $\upkappa_i$. We say that such solutions have asymptotics {\it prescribed} by $\vec\upkappa$. Then in \cite{Tian_Xu}, we proved that ${\mc M}_{\vec{P}}\left( \vec{\mc C}, B, \vec\upkappa \right)$ is the zero locus of a Fredholm section of certain Banach space bundle ${\mc E}$ over some Banach manifold ${\mc B}$. Moreover, the index of the Fredholm section is given by
\begin{align}\label{equation34}
\chi\left( \vec{\mc C}, B \right) = (2-2g) {\rm dim}_{\mb C} X_0 + 2 \left( c_1^G (B) -  \sum_{j=1}^m \Theta_j \right)- \sum_{j=1}^b {\rm dim}_{\mb C} X_{0, \upgamma_i}.
\end{align}
Here $c_1^G$ is the $G= S^1 \times S^1$-equivariant first Chern class of $X$, $\Theta_j \in {\mb Q}$ corresponds to certain degree shifting in Chen-Ruan cohomology. We remark that in the case of quasihomogeneous polynomials on ${\mb C}^N$, the above index coincides with the Fredholm index of the Witten equation calculated in \cite[Section 5]{FJR3}.

\subsection{The correlation function}

The correlation function we considered is a collection of multi-linear maps
\begin{align}\label{equation35}
\left\langle \ \cdot, \cdots, \cdot\ \right\rangle_{\vec{\mc C}}^B: \bigotimes_{i=1}^n {\ms H}_{\upgamma_i} \to {\mb Q},\ B \in \Gamma_X^G.
\end{align}
We can extend it trivially to a multi-linear map
\begin{align*}
\left\langle \ \cdot, \cdots, \cdot\ \right\rangle_{\vec{\mc C}}^B: \bigotimes_{i=1}^n {\ms H}_Q \to {\mb Q}.
\end{align*}
To define (\ref{equation35}), we take a strongly regular perturbation $\vec{P}= \left( \vec{a}, \vec{F} \right)$. Consider all possible combinations $\vec{\upkappa} = \left( \upkappa_1, \ldots, \upkappa_b \right) \in {\rm Crit} W_{\vec{P}}$ and the moduli space (\ref{equation33}). We claim 
\begin{thm}\cite{Tian_Xu_3}\label{thm35}
If $\vec{P}$ is strongly regular, then there exists a virtual fundamental class 
\begin{align*}
\left[ {\mc M}_{\vec{P}} \left( \vec{\mc C}, B, \vec{\upkappa} \right) \right]^{vir} \in H_{\chi\left( \vec{\mc C}, B \right)} \left( {\mc M}_{\vec{P}} \left( \vec{\mc C}, B, \vec{\upkappa} \right); {\mb Q}\right)
\end{align*}
\end{thm}
So we have the virtual counts $\# {\mc M}_{\vec{P}} \left( \vec{\mc C}, B, \vec\upkappa \right)\in {\mb Q}$, which is zero if $\chi(\vec{\mc C}, B) \neq 0$. Certain linear combination of the virtual numbers gives the correlation function. The coefficients of the linear combination are described as follows.

Consider the negative  gradient flow of the real part of $F_i$ restricted to $Q_{\upgamma_i}^{a_i} \subset X_{0, \upgamma_i}$, whose equilibria are all the $\upkappa_i$'s. Abbreviate $n_i= n_{\upgamma_i}$. Denote by 	
\begin{align*}
\left[ \upkappa_i^- \right] \in H_{n_i -1} \left( Q_{\upgamma_i}^{a_i}, F_i^{-\infty} \right)\ \left({\rm resp.}\ \left[ \upkappa_i^+ \right] \in H_{n_i -1} \left( Q_{\upgamma_i}^{a_i}, F_i^{\infty}\right) \right)
\end{align*}
the class of the unstable (resp. stable) manifold of this flow. Here 
\begin{align*}
F_i^{\infty}= Q_{\upgamma_i}^{a_i} \cap \left( {\rm Re} F_i \right)^{-1} \left( [M, +\infty) \right),\ F_i^{-\infty}= Q_{\upgamma_i}^{a_i} \cap \left( {\rm Re} F_i \right)^{-1} \left( (-\infty, -M] \right)
\end{align*}
for some $M>>0$. We still use $\left[ \upkappa_i^\pm \right]$ to denote their images under the map
\begin{align*}
H_{n_i - 1} \left( Q_{\upgamma_i}^{a_i}, F_i^{\pm\infty} \right) \to H_{n_i-1} \left( Q_{\upgamma_i}^{a_i}, \infty \right).
\end{align*}
To define (\ref{equation35}), we choose the last $n$ inputs (narrow states) to be the generators of the corresponding sectors $\theta_i= e_{\upgamma_i} \in {\ms H}_{\upgamma_i}$, $i = b+1, \ldots, b+n$. Suppose the first $b$ inputs (the broad states) are $\theta_i\in {\ms H}_{\upgamma_i}$, $i = 1, \ldots, b$. Then define
\begin{align}\label{equation36}
\left\langle \theta_1, \ldots, \theta_b, \theta_{b+1}, \ldots, \theta_m \right\rangle_{\vec{\mc C}}^B: = \sum_{\vec{\upkappa}} \# {\mc M}\left( \vec{\mc C}, B, \vec\upkappa \right) \prod_{i=1}^b \theta_i^* \cap \left[ \upkappa_i^- \right].
\end{align}
Here $\theta_i^*$ is the image of $\theta_i$ under (\ref{equation24}) and the $\cap$ is the intersection mentioned in (\ref{equation25}). In general (\ref{equation35}) is defined by taking linear extension of the above values.

\begin{rem}
In the future we would like to define descendant version of the correlation function. For this purpose we have to consider the variation of complex structures of the domain $\vec{\mc C}$. The moduli space of genus $g$, $m$-marked stable rigidified $r$-spin curve is a branched cover
\begin{align*}
\ov{\mc M}_{g, m}^r \to \ov{\mc M}_{g, m}
\end{align*}
over the Deligne-Mumford space (see \cite[Section 2]{FJR2}). We can consider the universal moduli space 
\begin{align*}
{\mc M}_{g, m}\left( B, \vec{\upkappa} \right)
\end{align*}
consists of gauge equivalence classes of solutions to all smooth rigidified $r$-spin curve of genus $g$ and $m$-marked points. We have to prove an extension of the compactness theorem of \cite{Tian_Xu} in which one allows the complex structure of the domain to vary and degenerate. In particular, when the complex structure degenerates, near the forming node the area form used for the vortex equation is exponentially small (in cylindrical coordinates); then we will be in a situation similar to what is considered in \cite{Mundet_Tian_2009}. When a broad node is forming, we have to include a strongly regular perturbation nearby as did in \cite{FJR2}. Nevertheless, we assume the existence of a good compactification of ${\mc M}_{g, m} \left( B, \vec{\upkappa} \right)$, denoted by $\ov{\mc M}_{g, m} \left( B, \vec{\upkappa} \right)$. We assume that the compactification has a virtual fundamental class
\begin{align*}
\left[ \ov{\mc M}_{g, m}(B, \vec\upkappa) \right]^{vir}
\end{align*}
whose degree is $6g-6$ more than the index in (\ref{equation34}). Then by pulling back cohomology classes of the Deligne-Mumford space via the forgetful map, we can evaluate them against the above virtual fundamental class. So the descendant invariants are defined.
\end{rem}

\section{Invariance of the correlation function}\label{section4}

In this section we list the properties of the fundamental virtual class given in Theorem \ref{thm35} should have, which will imply that the correlation functions are independent of the strongly regular perturbation $\vec{P}$. In the scope of the current series we only have to consider zero or one dimensional moduli spaces, so the properties can be stated in terms of the virtual counts $\# {\mc M}_{\vec{P}} \left( \vec{\mc C}, B, \vec\upkappa \right)$.

We briefly describe our argument. Suppose we have two strongly regular perturbation $\vec{P}_1$ and $\vec{P}_2$. It suffices to consider the case that $\vec{P}_1$ and $\vec{P}_2$ only differ at one broad puncture. Therefore we omit the dependence on perturbations at other punctures and suppose at this puncture, the monodromy is $\upgamma \in {\mb Z}_r$ and the two $\upgamma$-strongly regular perturbations are $P_0 = (a_0, F_0)$ and $(a_1, F_1)$. Then in Subsection \ref{subsection41}, using a homotopy argument, we show that there is another $\upgamma$-strongly regular perturbation $(a_1, F_0')$ for which the correlation functions defined by $(a_0, F_0)$ and $(a_1, F_0')$ are equal (indeed the corresponding virtual counts are equal). Therefore it remains to consider the case that $P_0$ and $P_1$ only differ in $F$. Then in Subsection \ref{subsection42} we show that the correlation functions defined for different $F$'s are equal. This is more complicated than the case considered in Subsection \ref{subsection41} because certain wall-crossing may happen during a homotopy of the perturbations. 

We remark that both parts of the argument rely on constructing Kuranishi structures (with boundaries) on certain 1-dimensional moduli spaces parametrized by homotopies of the perturbation terms. The details are given in \cite{Tian_Xu_3}.

\subsection{Independence of $a$}\label{subsection41}

\subsubsection*{Independence of the axial part}

We consider a $\upgamma$-strongly regular perturbation $P = (a, F)$. For any $\lambda >0$, $P_\lambda := \left( \lambda^r a, \lambda^r F_\lambda \right)$ is also strongly regular (cf. Lemma \ref{lemma32}). Moreover, the variation of $\lambda$ gives a homotopy $(P_t)_{t\in [0,1]}$ between $P$ and $P_\lambda$ and each $P_t$ is a $\upgamma$-strongly regular perturbation. Let $\vec{P}_t$ be the path of strongly regular perturbations for which the perturbations at all other broad punctures are fixed. They for each $\vec\upkappa \in {\rm Crit} W_{\vec{P}}$, the homotopy produces a smooth family $\vec\upkappa_t\in {\rm Crit} W_{\vec{P}_t}$. We consider the universal moduli space parametrized by this homotopy, denoted by 
\begin{align*}
\cup_{t \in [0,1]} {\mc M}_{\vec{P}_t} \left( \vec{C}, B, \vec\upkappa_t \right).
\end{align*}
We can construct a Kuranishi structure with boundary on the above moduli space, where the boundary contributes to the difference of the correlation functions. Since each $\vec{P}_t$ is strongly regular, the oriented boundary is 
\begin{align*}
{\mc M}_{\vec{P}_1} \left( \vec{C}, B, \vec\upkappa_1 \right) - {\mc M}_{\vec{P}_0} \left( \vec{C}, B, \vec\upkappa_0 \right).
\end{align*}
So the correlation functions defined by $\vec{P}_0$ and $\vec{P}_1$ are equal.

\subsubsection*{Independence of the angular part} Now suppose $P = (a, F)$ is $\upgamma$-strongly regular. Then 
\begin{align*}
P' = \left( e^{{\bm i} \alpha} a, F \circ e^{-{{\bm i} \alpha\over r}} \right)
\end{align*}
is also $\upgamma$-strongly regular. Let $\vec{P}$ and $\vec{P}'$ be the two strongly regular perturbations we want to compare, which coincide for every other broad puncture except for $P$ and $P'$. Then for each $\vec\upkappa\in {\rm Crit} W_{\vec{P}}$, there is a corresponding $\vec\upkappa' \in {\rm Crit} W_{\vec{P}'}$. 

Choose a smooth gauge transformation $g^\alpha: \Sigma^* \to S^1 \subset G$ which is equal to $e^{{\bm i}{\alpha \over r}}$ near $z_1$ and equal to the identity away from a neighborhood of $z_1$. It is easy to see

\begin{lemma} For each $B \in \Gamma_X^G$ and each $\vec\upkappa\in {\rm Crti} W_{\vec{P}}$, the map $(A, u, \vec\uppsi)\mapsto ( ( g^\alpha)^* A, (g^\alpha)^* u, \uppsi)$ induces an orientation-preserving homeomorphism
\begin{align*}
{\mc M}_{\vec{P}} \left( \vec{\mc C}, B, \vec\upkappa \right)\to {\mc M}_{\vec{P}'}\left( \vec{\mc C}, B, \vec\upkappa'  \right).
\end{align*}
\end{lemma}
Therefore, the corresponding virtual counts are equal. Moreover, $e^{{{\bm i} \alpha \over r}}$ induces a biholomorphism $Q_{\upgamma}^a\simeq Q_\upgamma^{e^{{\bm i} \alpha}a}$. The induced isomorphism 
\begin{align*}
\wt{H}_* \left( Q_\upgamma^a \right) \simeq \wt{H}_* \left( Q_\upgamma^{e^{{\bm i} \alpha} a} \right)
\end{align*}
is compatible with the isomorphisms (\ref{equation24}) for $a$ and $e^{{\bm i} \alpha} a$. $e^{{\bm i}\alpha \over r}$ also induces a one-to-one correspondence between the $\infty$-relative cycles in $Q_\upgamma^a$ and $Q_\upgamma^{e^{{\bm i} \alpha} a}$. Therefore, the coefficients in the linear combinations defining the correlation function are invariant. Therefore, the correlation functions defined for $\vec{P}$ and $\vec{P}'$ are equal.

\subsection{Independent of the choice of strongly regular $F$}\label{subsection42}

Now we need to compare two strongly regular perturbations which only differ at one broad puncture as $(a, F^0)$ vs. $(a, F^1)$. A generic homotopy $F^t$ connecting $F^0$ and $F^1$ may not be always strongly regular. For certain values of $t$ where the strong regularity is lost, wall-crossing happens. We first discuss the wall-crossing phenomenon in a general case.

\subsubsection*{BPS solitons and intersection of vanishing cycles}

Let $M$ be a noncompact K\"ahler manifold of complex dimension $m$ and $F: M \to {\mb C}$ be a holomorphic Morse function, which has finitely many critical points, listed as $\upkappa_1, \ldots, \upkappa_s$. If ${\rm Im} F(\upkappa_i)$ are distinct, we say that $F$ is strongly regular. In this case the unstable manifold of $\upkappa_i$ under the negative gradient flow of ${\rm Re} F$ defines a relative cycle 
\begin{align*}
\left[ \upkappa_i^- \right] \in H_m \left( M, F^{-\infty} \right).
\end{align*}
More generally, if we have a path $\gamma$ connecting $F(\upkappa_i)$ with a regular value $a$ of $F$ such that the path avoids singular values except $F(\upkappa_i)$, then there is a well-defined vanishing cycle
\begin{align*}
\partial \left[ \upkappa_i^\gamma \right] \in H_{m-1} \left( F^a \right),
\end{align*}
which only depends on the homotopy class of such paths.

Now suppose we have a homotopy $F^\upnu$, $\upnu\in [0,1]$ between two strongly regular holomorphic Morse functions $F^0, F^1$ such that $F^\upnu$ is a holomorphic Morse function for every $\upnu$. Then there are continuous curves $\upkappa_i^\upnu\in M$ such that 
\begin{align*}
\left\{ \upkappa_i^\upnu\ |\ i = 1, \ldots, s \right\} = {\rm Crit} F.
\end{align*}
On the other hand, there are canonical identifications
\begin{align*}
H_* \left( M, \left( F^0 \right)^{-\infty} \right) \simeq H_* \left( M,  \left( F^1 \right)^{-\infty} \right). 
\end{align*}
This is because the critical values of $F^\upnu$ are uniformly bounded. We denote the space in common as $H_* \left( M, F^{-\infty} \right)$. We would like to compare $\left[ \left( \upkappa_i^0 \right)^-\right]$ with $\left[ \left( \upkappa_i^1\right)^- \right]$ as elements of $H_* \left( M, F^{-\infty} \right)$. 

\begin{defn}
Suppose $F^0$ and $F^1$ are strongly regular holomorphic Morse functions on $M$ which are in the same path-connected components of the space of holomorphic Morse functions. A homotopy $F^\upnu$ ($\upnu \in [0,1]$) in the space of holomorphic Morse functions  is called {\bf strongly regular} if there exists $\upnu_1, \ldots, \upnu_k \in (0,1)$ such that
\begin{enumerate}
\item $F^\upnu$ is strongly regular for $\upnu \in [0,1]\setminus \{\upnu_1, \ldots, \upnu_k \}$. 

\item For each $j \in \{1, \ldots, k\}$, 
\begin{align*}
\# \left\{ {\rm Im} F^{\upnu_j} ( \upkappa_i^{\upnu_j})\ |\ i =1, \ldots, s \right\} = s-1,
\end{align*}
and there exist $i_j^-, i_j^+\in \{1, \ldots, s\}$, $\delta>0$ such that 
\begin{align*}
\forall \upnu \in ( \upnu_j- \delta, \upnu_j +\delta),\ {\rm Re} F^\upnu \left( \upkappa_{i_j^-}^\upnu \right) < {\rm Re} F^\upnu \left( \upkappa_{i_j^+}^\upnu \right),  
\end{align*}
and for $\upnu \in (\upnu_j - \delta, \upnu_j)$ and $\upnu \in (\upnu_j, \upnu_j + \delta)$, ${\rm Im} F^\upnu \left( \upkappa_{i_j^-}^\upnu \right) - {\rm Im} F^\upnu \left( \upkappa_{i_j^+}^\upnu \right)$ are of different signs.
\end{enumerate}

Each $\upnu_j$ is called a {\bf crossing} in this homotopy and we say that this crossing {\it happens} between $i_j^-$ and $i_j^+$. We say the crossing is positive (resp. negative), denoted by ${\rm sign} \upnu_j = 1$ (resp. ${\rm sign} \upnu_j = -1$), if the argument of $F^\upnu \left( \upkappa_{i_j^+}^\upnu \right) - F^\upnu\left( \upkappa_{i_j^-}^\upnu \right)$ rotates in the counterclockwise (resp. clockwise) direction as $\upnu$ moves from $\upnu_j - \delta$ to $\upnu_j + \delta$.
\end{defn}

It is easy to see that we can obtain a strongly regular homotopy by perturbation. Then to compare $\left[ \left( \upkappa_i^0 \right)^-\right]$ with $\left[ \left( \upkappa_i^1\right)^- \right]$, it suffices to consider the case that there is only one crossing at $\upnu = {1\over 2}$ in the homotopy (in the case of zero crossing, the two relative cycles are equal). In this case, we use $\wt{F}$ to denote the homotopy $\{F^\upnu\}$ and use $(-1)^{\wt{F}}$ to denote the sign of the only crossing. Suppose $i^-$ and $i^+$ are the two indices between which the crossing happens. Then we have the following Picard-Lefschetz formula (see \cite[Chapter 2]{Singularity_I}).
\begin{thm}\label{thm43}[Picard-Lefschetz]
For each $i \in \{1, \ldots, s\}$, we have
\begin{align}\label{equation41}
\left[ \left( \upkappa_i^1 \right)^- \right] - \left[ \left( \upkappa_i^0 \right)^- \right] = \left( -1 \right)^{\wt{F}} \delta_{i, i^+}  \left\langle \partial \left[ \upkappa_{i^+}^{\gamma_-} \right], \partial \left[  \upkappa_{i^-}^{\gamma_+} \right] \right\rangle_a \left[ \left( \upkappa_{i^-}\right)^-\right]
\end{align}
Here $a$ is the mid point of $F^{1\over 2} \left( \upkappa_{i^-} \right)$ and $F^{1\over 2} \left( \upkappa_{i^+} \right)$, $\gamma_+$ (resp. $\gamma_-$) is the straight path connecting $F^{1\over 2} \left( \upkappa_{i^-} \right)$ (resp. $F^{1\over 2} \left( \upkappa_{i^+} \right)$) to $a$, and $\langle \cdot, \cdot \rangle_a$ means the intersection pairing in $F^a$.
\end{thm}

The intersection number appeared in the Picard-Lefschetz formula can be intepreted as the number of BPS solitons. A BPS soliton is a nonconstant, finite energy solution $x: {\mb R} \to M$ to the ODE
\begin{align*}
x'(t) + \nabla F(x(t)) = 0.
\end{align*}
Here $\nabla F$ is the gradient of the real part of $F$. Then $\upkappa_\pm:= \displaystyle \lim_{t \to \pm\infty} x(t)$ are necessarily critical points of $F$ and 
\begin{align}\label{equation42}
{\rm Im} F(\upkappa_+) = {\rm Im} F(\upkappa_-),\ {\rm Re} F(\upkappa_+) < {\rm Re} F(\upkappa_-).
\end{align}
We identify two BPS solitons if they differ by a time translation. Then if (\ref{equation42}) is satisfied and other critical values of $F$ have different imaginary part, then the number of BPS solitons between $\upkappa_-$ and $\upkappa_+$ is finite and is equal to the intersection number appeared in (\ref{equation41}).

\subsubsection*{Wall-crossing formula for the virtual counts}

Now we consider two strongly regular perturbations $\vec{P}^\pm = \left( \vec{a}, \vec{F}^\pm \right)$ where $\vec{F}^\pm$ only differ at one broad puncture $z_1$ as $F^-_1$ and $F^+_1$, whose monodromy is denoted by $\upgamma_1$. We consider smooth homotopies which connect $F^-_1$ and $F^+_1$. By Hypothesis \ref{hyp33}, the space of $F$ for which $(a_1, F_1)$ is $\upgamma_1$-regular is path-connected. Therefore we can find a path $\wt{F}_1 = \left\{F_1^\upnu\right\}_{\upnu \in [-1, 1]}$ such that for each $\upnu\in [-1, 1]$, $(a_1, F_1^\upnu)$ is $\upgamma$-regular. Moreover, it suffices to consider the case that $\left( a_1, F_1^\upnu \right)$ is $\upgamma_1$-strongly regular for all $\upnu$ except for $\upnu= 0$. Such a homotopy induces a homotopy $\vec{P}^\upnu$, and families $\vec\upkappa^\upnu\in {\rm Crit} W_{\vec{P}^\upnu}$. 

\begin{thm}
If $(a_1, F^0_1)$ is strongly regular (i.e., there is no crossing), then 
\begin{align*}
\# {\mc M}_{\vec{P}^-} \left( \vec{C}, B, \vec\upkappa^{-1} \right) = \# {\mc M}_{\vec{P}^+} \left( \vec{C}, B, \vec\upkappa^{+1} \right).
\end{align*}
\end{thm}
The proof is a similar homotopy argument as used in Subsection \ref{subsection41}. Note that the $\infty$-relative cycles persist under the homotopy in this case, hence have same intersection numbers with any broad states. Therefore in the case of no crossing, the correlation functions defined on the two sides of the homotopy are equal.

Now we consider the case that a crossing happens at $\upnu = 0$, between $ \left( \upkappa_1^\upnu\right)_-$ and $\left( \upkappa_1^\upnu\right)_+$, with 
\begin{align*}
{\rm Re} F^\upnu \left( \upkappa_-^\upnu \right) < {\rm Re} F^\upnu \left( \upkappa_+^\upnu \right),\ \forall \upnu \in [-1,1].
\end{align*}
Then for each $\vec\upkappa^\upnu \in {\rm Crit} W_{\vec{P}^\upnu}$, we denote by $\vec\upkappa_\pm^\upnu \in {\rm Crit} W_{\vec{P}^\upnu}$ the asymptotic data obtained by replacing $\upkappa_1^\upnu$ by $\left(\upkappa_1^\upnu \right)_\pm$. Then we have the following wall-crossing formula.
\begin{thm}\label{thm45} For any family $\vec\upkappa^\upnu \in {\rm Crit} W_{\vec{P}^\upnu}$, we have
\begin{multline*}
\# {\mc M}_{\vec{P}^+} \left( \vec{\mc C}, B, \vec\upkappa^{+1} \right) - \# {\mc M}_{\vec{P}^-} \left( \vec{\mc C}, B, \vec\upkappa^{-1} \right)\\
= - (-1)^{\wt{F}_1} \delta_{\vec\upkappa, \vec\upkappa_-} \cdot \#_{BPS}\left( \left( \upkappa_1^0 \right)_-, \left( \upkappa_1^0 \right)_+ \right)\cdot \# {\mc M}_{\vec{P}^-} \left( \vec{\mc C}, B, \vec\upkappa_+ \right).
\end{multline*}
Here $\delta_{\vec\upkappa^\upnu, \vec\upkappa_-^\upnu}$ is the Kronecker delta, and $\#_{BPS}\left( \left( \upkappa_1^0 \right)_-, \left( \upkappa_1^0 \right)_+ \right)$ is the (algebraic) counts of the number of BPS solitons in $Q_{\upgamma_1}^{a_1}$ for the function $F_1^0$ between the two critical points.
\end{thm}

The proof uses a cobordism argument and the details are given in \cite{Tian_Xu_3}. In LG A-model, a similar wall-crossing formula was proved in \cite[Theorem 6.16]{FJR3}. We consider the universal moduli space
\begin{align*}
{\mc N}:= \bigcup_{\upnu \in [-1,1]} {\mc M}_{\vec{P}^\upnu} \left( \vec{\mc C}, B, \vec\upkappa^\upnu \right).
\end{align*}
This space is not compact due to degeneration of solitons at the slice of $\upnu = 0$. The soliton appeared are connecting $\left( \upkappa_1^0 \right)_-$ and $\left( \upkappa_1^0 \right)_+$ and stable solutions with BPS solitons exist in a codimension 1 subset and stable solutions with non-BPS solitons exist in a higher codimensional subset. Therefore, in the virtual sense, 
\begin{align*}
\partial {\mc N} \simeq \left( \bigcup_{\upnu = -1,1}{\mc M}_{\vec{P}^\upnu} \left( \vec{\mc C}, B, \vec\upkappa^\upnu \right)\right) \cup \left( {\mc M}_{\vec{P}^0} \left( \vec{\mc C}, B, \left( \vec\upkappa^0 \right)_+ \right)\times {\mc M}_{BPS} \left( \left( \upkappa_1^0 \right)_-, \left( \upkappa_1^0 \right)_+ \right) \right).
\end{align*}
Here ${\mc M}_{BPS}$ is the moduli of BPS solitons. Taking care of the orientation of the boundary, Theorem \ref{thm45} can be proved.

One difference between the proof of Theorem \ref{thm45} is that the BPS soliton used to compactify ${\mc N}$ are solutions $x: {\mb R} \to X_{\upgamma_1}$ to the equation
\begin{align}\label{equation43}
x'(s)  + \nabla W_1 = 0
\end{align}
but not for maps into $Q^{a_1}_{\upgamma_1}$. However, if the function $F_1$ is small, then solutions to (\ref{equation43}) are geometrically very close to solutions to 
\begin{align*}
y'(s) + \nabla \left( F_1|_{Q_{\upgamma_1}^{a_1}} \right) = 0.
\end{align*}
(See \cite{Lagrange_multiplier} for detailed treatment about the adiabatic limit of gradient flows in real Morse theory). Therefore the algebraic counting of BPS solitons will be the same as the intersection number between cycles in $Q^{a_1}_{\upgamma_1}$.

Therefore, to compare the correlation functions defined for $\vec{P}^+$ and $\vec{P}^-$, we see that the wall-crossing term appeared in the change of the virtual counts given in Theorem \ref{thm45} and the wall-crossing term appeared in the change of the $\infty$-relative cycles given in Theorem \ref{thm43} cancel each other. Similar situation happens for the LG A-model correlation function (see \cite{FJR3}). Therefore the correlation functions on the two sides of the homotopy are equal.

\bibliography{symplectic_ref}
	
\bibliographystyle{amsalpha}

\end{document}